\documentclass{article}

\usepackage{indentfirst}
\usepackage{amsmath,amsfonts,amsthm,amssymb}
\usepackage{mathrsfs}
\usepackage{amscd}
\usepackage{hyperref}

\def\co{\colon\thinspace}
\DeclareMathAlphabet{\mathsfsl}{OT1}{cmss}{m}{sl}
\newcommand{\tensor}[1]{\mathsfsl{#1}}
\newcommand{\spinc}{{\mathrm{Spin}^c}}
\newcommand{\spincz}{{\mathrm{Spin}^c_0}}

\theoremstyle{definition}

\newtheorem{thm}{Theorem}[section]
\newtheorem{prop}[thm]{Proposition}
\newtheorem{lem}[thm]{Lemma}
\newtheorem{remark}[thm]{Remark}

\newtheorem{defn}[thm]{Definition}
\newtheorem{cor}[thm]{Corollary}

\newtheorem{conj}[thm]{Conjecture}

\begin{document}

\title{Fintushel--Stern knot surgery in torus bundles}

\author{
{\Large Yi Ni}\\{\normalsize Department of Mathematics, Caltech, MC 253-37}\\
{\normalsize 1200 E California Blvd, Pasadena, CA
91125}\\{\small\it Email\/:\quad\rm yini@caltech.edu}
}


\date{}
\maketitle

\begin{abstract}
Suppose that $X$ is a torus bundle over a closed surface with homologically essential fibers. Let $X_K$ be the manifold obtained by Fintushel--Stern knot surgery on a fiber using a knot $K\subset S^3$. We prove that $X_K$ has a symplectic structure if and only if $K$ is a fibered knot. The proof uses Seiberg--Witten theory and a result of Friedl--Vidussi on twisted Alexander polynomials.
\end{abstract}

\section{Introduction}

One important question in $4$--dimensional topology is to determine which smooth closed $4$--manifolds admit symplectic structures. There are some topological constructions of symplectic $4$--manifolds. For example, Thurston \cite{ThSympl} showed that most surface bundles over surfaces are symplectic, and Gompf \cite{GS} generalized this result to Lefschetz fibrations. On the other hand, there are obvious obstructions to the existence of symplectic structures from algebraic topology. Moreover, Taubes' results \cite{Taubes1,Taubes2} provide more constraints in terms of the Seiberg--Witten invariants of the $4$--manifold. 

However, very little obstruction to the existence of symplectic structures is known besides the above mentioned ones. For example, given a symplectic manifold $X$, a symplectic torus $T\subset X$ with $[T]^2=0$,  and a knot $K
\subset S^3$, Fintushel and Stern \cite{FS} introduced a construction called knot surgery to get a new manifold $X_K$. They showed that $X_K$ is symplectic if $K$ is fibered, and $X_K$ can  often be proven to be non-symplectic when the Alexander polynomial of $K$ is not monic. (See Section~\ref{sect:Sympl} for more details.) However, if the Alexander polynomial of $K$ is monic, the obstruction from Seiberg--Witten theory does not exclude the possibility that $X_K$ has a symplectic structure. Nevertheless, one can mention the following folklore conjecture.

\begin{conj}\label{conj:FS}
Suppose that $X^4$ is a closed $4$--manifold admitting a Lefschetz fibration whose regular fibers are tori. Let $T\subset X$ be a regular fiber of the fibration, and suppose that $[T]\ne0$ in $H_2(X;\mathbb R)$. (Hence $X$ is symplectic by \cite{ThSympl}.) Let $X_K$ be a manifold obtained by Fintushel--Stern knot surgery on $T$ using a knot $K\subset S^3$. Then $X_K$ has a symplectic structure if and only if $K$ is a fibered knot.
\end{conj}

As we remarked before, the ``if'' part of the above conjecture was proved by Fintushel and Stern. The most interesting case of Conjecture~\ref{conj:FS} is when $\pi_1(X\setminus T)$ (and hence $\pi_1(X)$ and $\pi_1(X_K)$) is trivial, as $X_K$ is then homeomorphic to $X$ by Freedman's theorem. In this case, the Lefschetz fibration of $X$ must contain singular fibers.
Our main result in this paper is the case of the above conjecture when $X$ is a genuine torus bundle, namely, there are no singular fibers in the Lefschetz fibration.

\begin{thm}\label{thm:Main}
Conjecture~\ref{conj:FS} is true when the Lefschetz fibration of $X^4$ is a torus bundle.
\end{thm}

Friedl and Vidussi \cite{FV_Ann} proved that a closed four-manifold $S^1\times N$ is symplectic if and only if $N$ is a surface bundle over $S^1$. Their result implies the special case of Theorem~\ref{thm:Main} when $X$ is a trivial torus bundle $T^2\times F=S^1\times(S^1\times F)$, where $F$ is a closed surface. Our proof uses a similar strategy as in \cite{FV_Ann}. Namely, If $X_K$ has a symplectic structure, then any finite cover of $X_K$ also has a symplectic structure. We can then use the constraints from Seiberg--Witten theory to study the existence of symplectic structures on finite covers of $X_K$. The Seiberg--Witten invariants of finite covers of $X_K$ can be expressed in terms of twisted Alexander polynomials of $K$. We then use a vanishing theorem for twisted Alexander polynomials due to Friedl--Vidussi \cite{FV_JEMS} to get our conclusion. Of course, this strategy works only if the fundamental group of the $4$--manifold we consider contains many finite index subgroups.

A major difference between \cite{FV_Ann} and our case is that any finite cover $\widetilde N\to N$ gives rise to a finite cover $S^1\times \widetilde N\to S^1\times N$, but the construction of finite covers of $X_K$ is not so obvious. The main technical part of this paper is devoted to constructing finite covers of $X_K$. We also need the full strength of the gluing theorem for Seiberg--Witten invariants along essential $T^3$ from \cite{TaubesTori}.


Throughout this paper, the manifolds we consider are all smooth and oriented. Suppose that $M$ is a submanifold of a manifold $N$, then $\nu(M)$ denotes a closed tubular neighborhood of $M$ in $N$, and $\nu^{\circ}(M)$ denotes the interior of $\nu(M)$.

This paper is organized as follows. In Section~\ref{sect:TwAlex} we will review the definition of twisted Alexander polynomials and state a vanishing theorem of Friedl--Vidussi \cite{FV_JEMS}.  In
Section~\ref{sect:SW} we will review the Seiberg--Witten invariants for $4$--manifolds with boundary consisting of copies of $T^3$, and state the gluing formula for Seiberg--Witten invariants when glued along essential tori. In Section~\ref{sect:Sympl} we will review several constructions of symplectic $4$--manifolds, and state the constraints on symplectic $4$--manifolds from Seiberg--Witten theory. In Section~\ref{sect:Cover}, we will analyze the topology of torus bundles and construct certain covers of $X_K$. Our main theorem will then be proved in Section~\ref{sect:Proof}.

\vspace{5pt}\noindent{\bf Acknowledgements.}\quad  We are grateful to Stefan Friedl, Tian-Jun Li and Stefano Vidussi for comments on an earlier version of this paper. The author was
partially supported by NSF grant numbers DMS-1103976, DMS-1252992, and an Alfred P. Sloan Research Fellowship.  


\section{Twisted Alexander polynomials}\label{sect:TwAlex}

Twisted Alexander polynomials were introduced by Xiao-Song Lin \cite{Lin} in 1990. Many authors \cite{JW,Wada,KL,Cha} have since generalized this invariant in various ways. We will follow the treatment in \cite{FV_AJM}.

Let $N$ be a compact $3$--manifold with $b_1(N)>0$,
 $$H=H(N)=H^2(N,\partial N)/\mathrm{Tors}\cong H_1(N)/\mathrm{Tors}.$$ Let $F$ be a free abelian group, and $\phi\in Hom(H,F)$. Then $\pi_1(N)$ acts on $F$ by translation via $\phi$. Let $\alpha\co\pi_1(N)\to GL(n,\mathbb Z)$ be a representation. Then there is an induced representation
$$\alpha\otimes\phi\co\pi_1(N)\to GL(n,\mathbb Z[F])$$ 
defined as follows. For $g\in\pi_1(N)$, $\alpha\otimes\phi(g)$ sends $\sum_{f\in F}a_ff\in (\mathbb Z[F])^n$ to $$\sum_{f\in F}(\alpha(g)(a_f))(f\phi(g)),$$
where each $a_f\in \mathbb Z^n$, and the elements in $F$ are written multiplicatively. Thus $(\mathbb Z[F])^n$ is a left $\mathbb Z[\pi_1(N)]$--module, whose left $\mathbb Z[\pi_1(N)]$ multiplication commutes with the right $\mathbb Z[F]$--module structure.

Let $\widetilde N$ be the universal cover of $N$, then $\pi_1(N)$ acts on the left of $\widetilde N$ as group of deck transformations. The chain group $C_*(\widetilde N)$ is a right $\mathbb Z[\pi_1(N)]$--module, with the right action defined via $\sigma\cdot g:=g^{-1}(\sigma)$. We can form the chain complex
$$C_*(\widetilde N)\otimes_{\mathbb Z[\pi_1(N)]}(\mathbb Z[F])^n,$$
and define $H_*(N;\alpha\otimes\phi)$ be its homology group, which is also a
$\mathbb Z[F]$--module. We call $H_1(N;\alpha\otimes\phi)$
the {\it (first) twisted Alexander module}.

Since $\mathbb Z[F]$ is Noetherian, $H_1(N;\alpha\otimes\phi)$ is a finitely generated module over $\mathbb Z[F]$. There exists a free resolution
$$
\begin{CD}
(\mathbb Z[F])^m@>f>>(\mathbb Z[F])^n@>>>H_1(N;\alpha\otimes\phi)@>>>0,
\end{CD}
$$
where $m,n$ are positive integers. We can always arrange that $m\ge n$. Let $\tensor A$ be an $n\times m$ matrix over $\mathbb Z[F]$ representing $f$.

\begin{defn}
The {\it twisted Alexander polynomial} of $(N,\alpha,\phi)$, denoted by
 $\Delta_{N,\phi}^{\alpha}$, is the greatest common divisor of all the $n\times n$ minors of $\tensor A$. It is well defined only up to multiplication by a unit in $\mathbb Z[F]$.
\end{defn}

When $F=H$ and $\phi$ is the identity on $H$, we simply write $\Delta_N^{\alpha}$. When $\alpha$ is the trivial representation to $GL(1,\mathbb Z)$, we omit the superscript $\alpha$. In particular, $\Delta_N\in \mathbb Z[H]$ is the usual Alexander polynomial of $N$. When $\alpha\co \pi_1(N)\to G$ is a representation into a finite group, we get an induced representation into $\mathrm{Aut}(\mathbb Z[G])$, which is denoted by $\alpha$ as well. In that case, the twisted Alexander polynomials are essentially determined by the untwisted Alexander polynomials of the covers of $N$ correponding to $\ker \alpha$. More precisely, we recall
\cite[Proposition~3.6]{FV_AJM}:

\begin{prop}[Friedl--Vidussi]\label{prop:AlexHom}
Let $N$ be a $3$--manifold with $b_1(N)>0$ and let $\alpha\co \pi_1(N)\to G$ be an epimorphism onto a finite group. Let $N_G$ be the covering space of $N$ corresponding to $\ker\alpha$. Let $\pi_*\co H(N_G)\to H(N)$ be the map induced by the covering map. Then $\Delta_N^{\alpha}$ and $\Delta_{N_G}$ satisfy the following relations:
\newline$\bullet$
If $b_1(N_G)>1$, then
$$\Delta_N^{\alpha}=\left\{
\begin{array}{ll}
\pi_*(\Delta_{N_G}), &\text{if }b_1(N)>1;\\
(a-1)^2\pi_*(\Delta_{N_G}), &\text{if }b_1(N)=1, \mathrm{im}\pi_*=\langle a\rangle.
\end{array}
\right.$$
$\bullet$ If $b_1(N_G)=1$, then $b_1(N)=1$ and
$$\Delta_N^{\alpha}=\pi_*(\Delta_{N_G}).$$
\end{prop}

Given $\phi\in H^1(N)$, we say $\phi$ is {\it fibered} if $\phi$ is dual to a fiber of a fibration of $N$ over $S^1$.
A key ingredient in this paper is
the following vanishing theorem of Friedl--Vidussi
\cite{FV_JEMS} concerning non-fibered cohomology classes.

\begin{thm}[Friedl--Vidussi]\label{thm:FV}
Let $N$ be a compact, orientable, connected $3$--manifold with (possibly empty) boundary consisting of tori. If $\phi\in H^1(N)$ is not fibered, then there exsits an epimorphism $\alpha\co \pi_1(N)\to G$ onto a finite group $G$ such that
$\Delta_{N,\phi}^{\alpha}=0.$
\end{thm}


\section{Seiberg--Witten invariants and gluing formula along essential tori}\label{sect:SW}

In this section, we will review the Seiberg--Witten theory for $4$--manifolds with boundary consisting of tori, and the gluing formula for cutting along essential tori. We will follow the treatment in \cite{TaubesTori}.

First, let us recall the usual Seiberg--Witten invariants for closed $4$--manifolds \cite{Witten}.
Given a closed, oriented, connected, smooth, $4$--manifold $X$ with $b_2^+(X)>0$, let $\spinc(X)$ be the set of Spin$^c$ structures on $X$.
One can define the Seiberg--Witten invariant $sw_X\co \spinc(X)\to \mathbb Z/\{\pm1\}$. The sign can be fixed by choosing an orientation on 
$$L_X=\Lambda^{\mathrm{top}}H^1(X;\mathbb R)\otimes\Lambda^{\mathrm{top}}H^{2+}(X;\mathbb R).$$
In order to construct $sw_X$, we need to start with a riemannian metric on $X$. It turns out that
$sw_X$ does not depend on the choice of the metric when $b_2^+(X)>1$. When $b_2^+(X)=1$, there are two chambers in the space of metrics corresponding to two orientations on $H_2^+(X;\mathbb R)$,
$sw_X$ only depends on the chamber the metric lies in.

From now on in this section, we assume that $X$ is a compact, oriented, connected, smooth $4$--manifold such that $\partial X$ is a (possibly empty) disjoint union of $T^3$, and there exists a cohomology class
$\varpi\in H^2(X;\mathbb R)$ whose pull-back is non-zero in the cohomology of each component of $\partial X$. When $\partial X=\emptyset$, we do not need such $\varpi$ to define $sw_X$, but we still assume the existence of $\varpi$ in order to state the gluing formula. Moreover, we assume $b_2^+(X)>0$ when $\partial X=\emptyset$.

Let $\spinc(X)$ be the set of Spin$^c$ structures on $X$, and $\spincz(X)\subset \spinc(X)$ be the subset consisting of $\mathfrak s$ such that the pull-back of $c_1(\mathfrak s)$ is zero in $H^2(\partial X)$. By the exact sequence 
$$
\begin{CD}
H^2(X,\partial X)@>\pi^*>> H^2(X)@>\iota^*>> H^2(\partial X),
\end{CD}
$$
if $\mathfrak s\in \spincz(X)$, then $c_1(\mathfrak s)\in H^2(X)$ is in the image of $\pi^*$. Let $$\spincz(X,\partial X)=\left\{(\mathfrak s, z)\left|\mathfrak s\in \spincz(X), z\in H^2(X,\partial X), \pi^*(z)=c_1(\mathfrak s)\right.\right\}.$$

One can define the relative Seiberg--Witten invariant 
$$sw_X\co \spincz(X,\partial X)\to\mathbb Z/\{\pm1\}.$$ The sign can be fixed by choosing an orientation on 
$$L_X=\Lambda^{\mathrm{top}}H^1(X,\partial X;\mathbb R)\otimes\Lambda^{\mathrm{top}}H^{2+}(X,\partial X;\mathbb R).$$
When $\partial X=\emptyset$, $sw_X$ is just the usual Seiberg--Witten invariant.
When $\partial X\ne\emptyset$, $sw_X$ is an invariant of the pair $(X,\varpi)$, and it is unchanged under continuous deformation of $\varpi$ in $H^2(X;\mathbb R)$ through classes with non-zero
restriction in the cohomology of each component of $\partial X$.

Suppose that $M\subset X$ is a $3$--torus such that the restriction of $\varpi$ to $H^2(M;\mathbb R)$ is nontrivial. 
We will consider the gluing formula for $sw$ when $X$ is cut open along $M$.
There are two cases. In the first case, $X$ is split by $M$ into two parts $X_1,X_2$. In the second case, $X_1=X\setminus\nu^{\circ}(M)$ is connected.

When $M$ is separating, there is a canonical isomorphism
\begin{equation}\label{eq:LX1}
L_X\cong L_{X_1}\otimes L_{X_2}.
\end{equation}
One can define a map
$$\wp\co\spincz(X_1,\partial X_1)\times \spincz(X_2,\partial X_2)\to\spincz(X,\partial X).$$

When $M$ is non-separating, 
there is a canonical isomorphism
\begin{equation}\label{eq:LX2}
L_X\cong L_{X_1}.
\end{equation}
One can define
$$\wp\co\spincz(X_1,\partial X_1)\to \spincz(X,\partial X).$$

In any case,
if $(\mathfrak s,z)\in\mathrm{im}\wp$, then $c_1(\mathfrak s)|_M=0$.

\begin{thm}[Taubes]\label{thm:Gluing1}
Let $M\subset X$ be a three-dimensional torus satisfying that the pull-back of $\varpi$ in $H^2(M;\mathbb R)$ is nontrivial. \newline
$\bullet$ If $M$ splits $X$ into two parts $X_1,X_2$, we orient $L_X$ using (\ref{eq:LX1}). Then
$$sw_X(\mathfrak s,z)=\sum_{((\mathfrak s_1,z_1),(\mathfrak s_2,z_2))\in\wp^{-1}(\mathfrak s,z)}sw_{X_1}(\mathfrak s_1,z_1)sw_{X_2}(\mathfrak s_2,z_2).$$
$\bullet$ If $M$ does not split $X$, let $X_1=X\setminus\nu^{\circ}(M)$, we orient $L_X$ using (\ref{eq:LX2}). Then
$$sw_X(\mathfrak s,z)=\sum_{(\mathfrak s_1,z_1)\in\wp^{-1}(\mathfrak s,z)}sw_{X_1}(\mathfrak s_1,z_1).$$
\end{thm}

Theorem~\ref{thm:Gluing1} implies the more general case of the gluing formula when we cut $X$ open along more than one tori. Suppose that $M=M_1\cup M_2\cup\cdots\cup M_m$ is a disjoint union of 
three-dimensional tori in $X$ such that the restriction of $\varpi$ to $H^2(M_i;\mathbb R)$ is nontrivial for every $i$. Let $X_1,\dots,X_n$ be the components of $X\setminus\nu^{\circ}(M)$. 
Let $\mathcal G$ be the graph with vertices $v_1,\dots,v_n$ and edges $e_1,\dots,e_m$. The incidence relation in $\mathcal G$ is as follows: if $M_k$ is adjacent to $X_i$ and $X_j$, the edge $e_k$ connects $v_i$ and $v_j$.
Let $\mathcal T$ be a spanning tree of $\mathcal G$, then $\mathcal T$ has exactly $n-1$ edges. Without loss of generality, we may assume the edges in $\mathcal G\setminus\mathcal T$ are $e_1,\dots,e_{m-n+1}$. We consider a sequence of manifolds $X^{(i)}$, $i=0,\dots,m$:
$$X^{(0)}=X,\quad X^{(i)}=X^{(i-1)}\setminus\nu^{\circ}(M_i), i>0.$$
Clearly, $M_i$ is non-separating in $X^{(i-1)}$ when $1\le i\le m-n+1$, and $M_i$ is separating in $X^{(i-1)}$ when $m-n+2\le i\le m$. Thus we can apply Theorem~\ref{thm:Gluing1} inductively to get the gluing formula when we cut open along $M$.

More precisely, applying (\ref{eq:LX2}) and (\ref{eq:LX1}) consecutively, we get a canonical isomorphism
$$
L_X\cong \bigotimes_{i=1}^n L_{X_i},
$$
which will be used to orient $L_X$.
We can also define a map
$$\wp\co\prod_{i=1}^n\spincz(X_i,\partial X_i)\to\spincz(X,\partial X).$$

\begin{thm}\label{thm:Gluing2}
Under the above settings, we have
$$sw_X(\mathfrak s,z)=\sum_{((\mathfrak s_1,z_1),\dots,(\mathfrak s_n,z_n))\in\wp^{-1}(\mathfrak s,z)}\prod_{i=1}^nsw_{X_i}(\mathfrak s_i,z_i).$$
\end{thm}

In practice, it is more convenient to consider the following version of Seiberg--Witten invariant:
$$SW_X\co H^2(X,\partial X)\to \mathbb Z$$
defined by letting
$$SW_X(z)=\sum_{(\mathfrak s,z)\in\spincz(X,\partial X)}sw_X(\mathfrak s,z).$$

Let $$\rho_i\co H^2(X_i,\partial X_i)\to H^2(X,\partial X), \quad i=1,\dots,n$$ be the natural maps, and 
$$\rho=\rho_1+\cdots+\rho_n\co \bigoplus_{i=1}^nH^2(X_i,\partial X_i)\to H^2(X,\partial X).$$

Then Theorem~\ref{thm:Gluing2} implies
\begin{thm}\label{thm:Gluing3}
Under the condition of Theorem~\ref{thm:Gluing2}, we have
$$SW_X(z)=\sum_{(z_1,\dots,z_n)\in\rho^{-1}(z)}\prod_{i=1}^nSW_{X_i}(z_i).$$
\end{thm}

It is often convenient to represent the Seiberg--Witten invariants in the following more compact form.

Let $H(X)=H^2(X,\partial X)/\mathrm{Tors}$.
Given $z\in H^2(X,\partial X)$,
let $[z]\in H(X)$ be the reduction of $z$.
We define
$$\underline{SW}_X=\sum_{z\in H^2(X,\partial X)}SW_X(z)[z],$$
which lies either in $\mathbb Z[H(X)]$, or, in certain cases, an extension of this group ring which allows semi-infinite power series.

For example, let $t\in H(D^2\times T^2)$ be the Poincar\'e dual to the fundamental class of the torus, then
\begin{equation}\label{eq:D2T2}
\underline{SW}_{D^2\times T^2}=\frac{t}{1-t^2}=t+t^3+\cdots.
\end{equation}

The invariant $\underline{SW}_X$ is related to the Alexander polynomial of a $3$--manifold.
Let
$N$ be a compact, oriented, connected $3$--manifold with $b_1(N)>0$ such that $\partial N$ is a (possibly empty) disjoint union of $T^2$.
Let $p^*\co H^2(N,\partial N)\to H^2(S^1\times N,\partial(S^1\times  N))$ be the map on cohomology induced by the projection $p\co S^1\times N\to N$. Let $$\Phi_2\co \mathbb Z[H(N)]\to \mathbb Z[H(S^1\times N)]$$ be the map induced by $2p^*$. Meng and Taubes \cite{MT} proved the following theorem.

\begin{thm}[Meng--Taubes]\label{thm:MT}
Let $N$ be a compact, oriented, connected $3$--manifold with $b_1(N)>0$ such that $\partial N$ is a (possibly empty) disjoint union of $T^2$. When $b_1(N)=1$, let $t$ be a generator of $H(N)\cong\mathbb Z$, and let $|\partial N|=0$ or $1$ be the number of boundary components of $\partial N$.  
Then, there exits an element $\xi\in\pm p^*( H(N))$, such that
$$\underline{SW}_{S^1\times N}=
\left\{\begin{array}{ll}
\xi\Phi_2(\Delta_N), &\text{if }b_1(N)>1;\\
\xi\Phi_2((1-t)^{|\partial N|-2}\Delta_N), &\text{if }b_1(N)=1.
\end{array}\right.
$$
\end{thm}

As a corollary, we prove the following gluing result for the Alexander polynomial.

\begin{cor}\label{cor:AlexGlue}
Let $N$ be as in Theorem~\ref{thm:MT}, and $K\subset N$ be a knot such that $[K]$ is nontorsion. Let $\kappa\in H[N]$ be the coset of the Poincar\'e dual of $[K]$. Let $M=N\setminus\nu^{\circ}(K)$, and let \[\pi^*\co H^2(M,\partial M)\cong H^2(N,\nu(K)\cup\partial N)\to H^2(N,\partial N)\] be the natural map induced by the inclusion $(N,\partial N)\subset (N,\nu(K)\cup\partial N)$. We also use $\pi^*$ to denote the induced map $\mathbb Z[H(M)]\to \mathbb Z[H(N)]$. Then there exists an element $\xi\in \pm H(N)$, such that
$$\Delta_{N}=\left\{\begin{array}{ll}
\xi\pi^*(\Delta_M), &\text{if }b_1(N)=1;\\
\xi(1-\kappa)^{-1}\pi^*(\Delta_M), &\text{if }b_1(N)>1.
\end{array}\right.$$
\end{cor}
\begin{proof}
Let $p_N^*\co H^2(N,\partial N)\to H^2(S^1\times N,S^1\times\partial N)$, and define $p_M^*$ similarly.
We first consider the case $b_1(N)>1$. By Theorem~\ref{thm:MT}, there exist $\zeta\in\pm p_N^*(H(N))$ and $\eta=\pm p_M^*(H(M))$ such that
\begin{equation}\label{eq:DeltaN}
\zeta\Phi_2(\Delta_N)=\sum_{z\in H^2(N,\partial N)}SW_{S^1\times N}(p_N^*(z))[p_N^*(z)],
\end{equation}
and
\begin{equation}\label{eq:DeltaM}
\eta\Phi_2(\Delta_M)=\sum_{w\in H^2(M,\partial M)}SW_{S^1\times M}(p_M^*(w))[p_M^*(w)].
\end{equation}
Let $a\in H^2(T^2\times D^2,T^2\times\partial D^2)$ be the positive generator.
Using Theorem~\ref{thm:Gluing3} and (\ref{eq:D2T2}), we get
\begin{eqnarray*}
&&SW_{S^1\times N}(p_N^*(z))\\
&=&\sum_{n\in\mathbb Z}SW_{T^2\times D^2}((2n+1)a)\sum_{\substack{w\in H^2(M,\partial M)\\ \rho((2n+1)a),p_M^*(w))=p_N^*(z)}}SW_{S^1\times M}(p_M^*(w))\\
&=&\sum_{n\ge0}\sum_{\substack{w\in H^2(M,\partial M)\\ \rho((2n+1)a,p_M^*(w))=p_N^*(z)}}SW_{S^1\times M}(p_M^*(w)).
\end{eqnarray*}
Using  (\ref{eq:DeltaN}), (\ref{eq:DeltaM}), and the fact that 
$$\rho((2n+1)a,p_M^*(w))=p_N^*((2n+1)\mathrm{PD}([K])+\pi^*(w)),$$
we get
\begin{eqnarray*}
&&\zeta\Phi_2(\Delta_N)\\
&=&\sum_{z\in H^2(N,\partial N)}\sum_{n\ge0}\sum_{\substack{w\in H^2(M,\partial M)\\ \rho((2n+1)a,p_M^*(w))=p_N^*(z)}}SW_{S^1\times M}(p_M^*(w))[p_N^*(z)]\\
&=&\sum_{n\ge0}\sum_{w\in H^2(M,\partial M)}SW_{S^1\times M}(p_M^*(w))\kappa^{2n+1}[\pi^*(w)]\\
&=&\frac{\kappa}{1-\kappa^2}\pi^*(\eta\Phi_2(\Delta_M)).
\end{eqnarray*}
So our result holds.

When $b_1(N)=1$, the proof is similar.
\end{proof}


\section{Symplectic geometry}\label{sect:Sympl}

In this section, we will review some topological constructions of symplectic $4$--manifolds, and state the constraints on the Seiberg--Witten invariants of symplectic manifolds.
 
Thurston \cite{ThSympl} found a very general topological construction of symplectic manifolds:

\begin{thm}[Thurston]\label{thm:ThSympl}
Let $M^{2n+2}\to N^{2n}$ be a fiber bundle over a symplectic manifold. If the homology class of the fiber is nonzero in $H_2(M;\mathbb R)$, then $M$ has a symplectic structure such that each fiber is a symplectic submanifold. Moreover, if $\rho\co N\hookrightarrow M$ is a section, then the image of $\rho$ is a symplectic submanifold.
\end{thm}

In dimension $4$, Thurston's construction was generalized by Gompf \cite{GS} to the extent that if a $4$--manifold $X$ admits a Lefschetz fibration (or a Lefschetz pencil) such that the homology class of the fiber is nontorsion, then $X$ has a symplectic structure. This construction, together with the celebrated theorem of Donaldson \cite{Donaldson} that all closed symplectic manifolds have Lefschetz pencils, gives us a topological characterization of closed symplectic $4$--manifolds.

The above characterization of symplectic $4$--manifolds is not always practical. When we construct symplectic $4$--manifolds, we often need
the following construction due to Gompf \cite{Gompf} and McCarthy--Wolfson \cite{MW}.
Suppose that $X_1,X_2$ are two smooth four-manifolds, $F_i\subset X_i$, $i=1,2$, are two $2$--dimensional closed connected submanifolds such that $F_1$ is homeomorphic to $F_2$ and $[F_1]^2=-[F_2]^2$. Let $N(F_i),\nu(F_i)$ be two tubular neighborhoods of $F_i$ in $X_i$, $i=1,2$, such that $\nu(F_i)$ is contained in the interior of $N(F_i)$. Let $W_i=N(F_i)\setminus\nu^{\circ}(F_i)$, $i=1,2$, regarded as an annulus bundle over $F_i$. Suppose that $f\co F_1\to F_2$ is a diffeomorphism, then there exists an orientation preserving diffeomorphism $\bar f\co W_1\to  W_2$
such that $\bar f(\partial N(F_1))=\partial \nu(F_2)$, and $\bar f$ is a bundle map covering $f$. Let $X$ be the manifold obtained by gluing $X_1\setminus\nu^{\circ}(F_1)$ and $X_2\setminus \nu^{\circ}(F_2)$ together via the diffeomorphism $\bar f$.  Then $X$ is called the {\it normal connected sum} of $(X_1,F_1)$ and $(X_2,F_2)$, denoted $X_1\#_fX_2$. If $X_i$ is symplectic, $F_i$ is a symplectic submanifold, $i=1,2$, and $f,\bar f$ are chosen to be symplectomorphisms, then $X$ also has a symplectic structure, and the operation is called a {\it symplectic normal connected sum} or simply {\it symplectic sum}. 

Suppose that $X$ is a smooth $4$--manifold containing a smooth $2$--torus $T$ with $[T]^2=0$. Let $K\subset S^3$ be a knot, and let $K'\subset S^3_0(K)$ be the dual knot in the zero surgery. We can perform the normal connected sum of $(X,T)$ and $(S^1\times S^3_0(K),S^1\times K')$ to get a new manifold $X_K$. (This $X_K$ is usually not unique, since it depends on the choice of a homeomorphism $f$ and $\bar f$.) This procedure was investigated by Fintushel and Stern \cite{FS}, who called it {\it knot surgery}. By Theorems~\ref{thm:Gluing1} and \ref{thm:MT}, we know that 
\begin{equation}\label{eq:FS}
\underline{SW}_{X_K}=\underline{SW}_X\cdot\Delta_K(\mathrm{PD}([T])^2),
\end{equation}
where $\Delta_K$ is the Alexander polynomial of $K$. (Clearly, $X_K$ has the same homology type as $X$, so we can identify $H(X_K)$ with $H(X)$.) This construction is particularly interesting when $\pi_1(X\setminus T)=1$, since $X_K$ is then homeomorphic to $X$ by Freedman's theorem, but $X_K$ is not diffeomorphic to $X$ if $\Delta_K\ne1$.

When $K$ is fibered, $S^1\times S^3_0(K)$ is a surface bundle over $T^2$ with $S^1\times K'$ being a section, and the fiber is homologically essential. Theorem~\ref{thm:ThSympl} implies that $S^1\times S^3_0(K)$ has a symplectic structure such that $S^1\times K'$ is a symplectic submanifold. Hence the symplectic sum construction implies the following theorem.

\begin{thm}[Fintushel--Stern]\label{thm:FS}
Suppose that $X$ is a symplectic $4$--manifold, $T\subset X$ is a symplectic torus with $[T]^2=0$. Then
$X_K$ is symplectic if $K$ is fibered.
\end{thm}

It is natural to guess that a converse to Theorem~\ref{thm:FS} should be true in many cases. More precisely, one can mention the folklore Conjecture~\ref{conj:FS}. Evidence to this conjecture is a famous theorem of Taubes
\cite{Taubes1,Taubes2}.

\begin{thm}[Taubes]\label{thm:Taubes}
Suppose that $(X,\omega)$ is a closed symplectic $4$--manifold with $b_2^+>1$, $\mathfrak k$ is the canonical Spin$^c$ structure on $X$, and $\bar{\mathfrak k}$ is the conjugate of $\mathfrak k$. Then $$SW_X(\mathfrak k)=\pm1.$$ Moreover, if $\mathfrak s\in\spinc(X)$ satisfies that $SW_X(\mathfrak s)\ne0$, then $$|c_1(\mathfrak s)\smile[\omega]|\le c_1(\mathfrak k)\smile[\omega],$$ and the equality holds if and only if $\mathfrak s=\mathfrak k$ or $\bar{\mathfrak k}$.
\end{thm}

In particular, if $X$ is the K3 surface, and $\Delta_K$ is not monic, Taubes' theorem implies that $X_K$ is not symplectic.

We will also need the following theorem proved by Bauer \cite{B} and Li \cite{L2}.

\begin{thm}[Bauer, Li]\label{thm:SymplCY}
Suppose that $X$ is a closed symplectic $4$--manifold with $c_1(\mathfrak k)$ torsion. Then $b_1(X)\le4$.
\end{thm}


\section{Constructing covering spaces of $X_K$}\label{sect:Cover}

Before we state the main result in this section, we set up the basic notations we will use.
Let $X$ be a torus bundle over a closed surface $F$. Let $T$ be a fiber of $X$, and let $E=X\setminus \nu^{\circ}(T)$. Let $K\subset S^3$ be a nontrivial knot,
 $N=S^3\setminus \nu^{\circ}(K)$, $N_0=S^3_0(K)$ be the zero surgery on $K$, and $K'\subset N_0$ be the dual knot of the surgery. Let $f\co S^1\times K'\to T$ be a diffeomorphism, and let $X_K=X\#_f(S^1\times N_0)$.

The goal in this section is to construct covering spaces of $X_K$. More precisely, we will prove the following proposition.

\begin{prop}\label{prop:Cover}
Suppose that $\alpha\co \pi_1(N_0)\to G$ is an epimorphism, where $G$ is a finite group.
Let $p_0\co\widetilde N_0\to N_0$ be the covering map corresponding to $\ker\alpha$, and let $\widetilde N=p_0^{-1}(N)$.
Suppose that $p_0^{-1}(K')$ has $r$ components. Since $p_0$ is a regular cover, the restriction of $p_0$ on each component of $p_0^{-1}(K')$ has the same degree $l$.
If the genus of $F$ is positive, then there exists an $rl^3$--fold cover $\widetilde X_K$ of $X_K$, 
such that $\widetilde X_K$ contains a submanifold diffeomorphic to $S^1\times \widetilde N$, and $\widetilde X_K$ admits a retraction onto the complete bipartite graph $K_{r,l}$.
\end{prop}

In order to prove this proposition, we will need some preliminary material.
We start by analyzing the topology of torus bundles. 
The structural group of a torus bundle is $\mathrm{Diff}^+(T^2)$, which is homotopy equivalent to its subgroup $\mathrm{Aff}^+(T^2)\cong T^2\rtimes \mathrm{SL}(2,\mathbb Z)$. If the structural group is contained in $\mathrm{SL}(2,\mathbb Z)$, we say this torus bundle is an {\it $\mathrm{SL}(2,\mathbb Z)$--bundle}.

Each torus bundle $X\to F$ is uniquely determined up to isomorphism by the homotopy type of its classifying map $F\to B\mathrm{Diff}^+(T^2)\simeq B\mathrm{Aff}^+(T^2)$. From the short split exact sequence
$$1\to T^2\to \mathrm{Aff}^+(T^2)\to \mathrm{SL}(2,\mathbb Z)\to 1$$
we get a fiber bundle 
$$BT^2\to B\mathrm{Aff}^+(T^2)\to B\mathrm{SL}(2,\mathbb Z)$$
which has a section. Since $BT^2=\mathbb CP^{\infty}\times \mathbb CP^{\infty}=K(\mathbb Z^2,2)$ and $B\mathrm{SL}(2,\mathbb Z)=K(\mathrm{SL}(2,\mathbb Z),1)$, we have
\begin{eqnarray*}
\pi_1(B\mathrm{Aff}^+(T^2))&\cong& \mathrm{SL}(2,\mathbb Z),\\
\pi_2(B\mathrm{Aff}^+(T^2))&\cong&\mathbb Z^2.
\end{eqnarray*}
Hence the homotopy type of a map $F\to B\mathrm{Aff}^+(T^2)$ is determined by
a representation $\rho\co \pi_1(F)\to\pi_1(B\mathrm{Aff}^+(T^2))\cong \mathrm{SL}(2,\mathbb Z)$ (called the {\it monodromy}) and a pair of integers $(m,n)\in H^2(F;\pi_2(B\mathrm{Aff}^+(T^2)))\cong\mathbb Z^2$ (called the {\it Euler class}). 

\begin{remark}\label{rem:PosGenus}
In particular, when $F=S^2$, $X$ is completely determined by the Euler class. It is easy to see $[T]\ne0\in H_2(X;\mathbb R)$ if and only if $(m,n)=(0,0)$. In this case, $X=T^2\times S^2$. As we mentioned before, this case is covered by Friedl and Vidussi's work \cite{FV_Ann}. Hence, in order to prove Theorem~\ref{thm:Main}, we only need to consider the case when the genus of $F$ is positive.
\end{remark}

\begin{remark}
In general, $[T]\ne0\in H_2(X;\mathbb R)$ if and only if $E^{\infty}_{0,2}\cong\mathbb Z$, where $\{E^{i}_{*,*}\}_{i=1}^{\infty}$ is the Leray--Serre spectral sequence for the fiber bundle $X\to F$. (See \cite[Section~4]{Geiges} or \cite[Lemma~4.6]{Walc} for more detail.) When $F$ is a torus, Geiges explicitly described the cases when $[T]\ne0$ \cite[Theorem~1]{Geiges}, using Sakamoto--Fukuhara's classification of torus bundles over torus \cite{SF}.
\end{remark}

\begin{defn}
Let $T\subset Y^4$ be a torus with trivial neighborhood. We fix a product structure $S^1\times S^1$ on $T$ and identify $S^1$ with $\mathbb R/\mathbb Z$. We can remove a neighborhood $\nu(T)\cong T^2\times D^2$ then glue it back to $Y\setminus\nu^{\circ}(T)$ via the homeomorphism $f\co T^2\times\partial D^2\to \partial(Y\setminus\nu^{\circ}(T))$ which sends $(x,y,\theta)$ to $(x+m\theta,y+n\theta,\theta)$. This procedure is called the
{\it $(m,n)$--framed surgery} on $T$. 
\end{defn}

Given $\rho$ and $(m,n)$, as in \cite[Section~4]{Walc}, we can reconstruct $X\to F$ by first constructing an $\mathrm{SL}(2,\mathbb Z)$--bundle over $F$ using the monodromy $\rho$ then doing $(m,n)$--framed surgery on a fiber.
Suppose that $$\pi_1(F)=\langle a_1,a_2,\dots,a_{2g-1},a_{2g}|\prod_{k=1}^g[a_{2k-1},a_{2k}]\rangle$$ and that $\rho(a)$ acts on $\mathbb Z^2=\langle s_1,s_2|[s_1,s_2]\rangle$ for every $a\in\pi_1(F)$. We can write down a presentation of $\pi_1(X)$ from the construction of $X$ as follows:
\begin{equation}\label{eq:Presentation}
\pi_1(X)=\left\langle s_1,s_2, t_1,\dots,t_{2g}\left| 
\begin{array}{l}
[s_1,s_2],\\
t_is_jt_i^{-1}(\rho(a_i)(s_j))^{-1}, (1\le i\le2g, j=1,2)\\
s_1^ms_2^n(\prod_{k=1}^g[t_{2k-1},t_{2k}])^{-1}
\end{array}
\right.\right\rangle.
\end{equation}

\begin{prop}\label{prop:TorusCover}
Let $X\to F$ be a torus bundle over a closed surface with positive genus. For any integer $l>0$, there exists a torus bundle $\widetilde X$ and an $l^3$--fold cover $p \co \widetilde X\to X$, such that for any fiber $T\subset X$ and any component $\widetilde T$ of $p^{-1}(T)$, the map $p|_{\widetilde T}\co \widetilde T\to T$ is the covering map corresponding to the characteristic subgroup $(l\mathbb Z)\times(l\mathbb Z)\subset \pi_1(T)$.
\end{prop}
\begin{proof}
Let $\overline F\to F$ be an $l$--fold cover, and $\overline X\to \overline F$ be a torus bundle over $\overline F$ which is the pull-back of $X\to F$. Suppose that the genus of $\overline F$ is $\bar g$, and the monodromy of $\overline X$ is $\bar{\rho}$. Suppose that the Euler class of $X\to F$ is $(m,n)$, then the Euler class of $\overline X\to \overline F$ is $(ml,nl)$. By (\ref{eq:Presentation}), 
$$
\pi_1(\overline X)=\left\langle s_1,s_2, t_1,\dots,t_{2\bar g}\left| 
\begin{array}{l}
[s_1,s_2],\\
t_is_jt_i^{-1}(\bar{\rho}(a_i)(s_j))^{-1}, (1\le i\le2\bar g, j=1,2)\\
s_1^{ml}s_2^{nl}(\prod_{k=1}^{\bar g}[t_{2k-1},t_{2k}])^{-1}
\end{array}
\right.\right\rangle.
$$
Let $\Gamma_l$ be the subgroup of $\Gamma=\pi_1(\overline X)$ generated by $s_1^l,s_2^l,t_1,\dots,t_{2\bar g}$, we claim that $[\Gamma:\Gamma_l]=l^2$. If this claim is true, let
$\widetilde X$ be the covering space of $\overline X$ corresponding to $\Gamma_l$, then $\widetilde X$ is the covering space of $X$ we want.

The rest of this proof is devoted to proving $[\Gamma:\Gamma_l]=l^2$. Any element in $\Gamma$ can be written as a word $st$, where $s$ is a word in $s_1^{\pm1},s_2^{\pm1}$, $t$ is a word in $t_i^{\pm1}$, $1\le i\le2\bar g$. Since the subgroup $\Sigma_l=\langle s_1^l, s_2^l\rangle$ of $\langle s_1, s_2\rangle\cong\mathbb Z^2$ is preserved by any $\bar{\rho}(a_i)$, 
$st\in\Gamma_l$ if and only if $s\in \Sigma_l$. Let $(u_1,v_1),(u_2,v_2)\in \{0,1,\dots,l-1\}^2$, then it follows that $s_1^{u_1}s_2^{v_1}\in s_1^{u_2}s_2^{v_2}\Gamma_l$ if and only if $(u_1,v_1)=(u_2,v_2)$.
So
\[
s_1^us_2^v\Gamma_l, \quad (u,v)\in\{0,1,\dots,l-1\}^2
\]
are distinct left cosets of $\Gamma_l$ in $\Gamma$. Clearly, the union of these cosets is $\Gamma$, so $[\Gamma:\Gamma_l]=l^2$.
\end{proof}

\begin{proof}[Proof of Proposition~\ref{prop:Cover}]
By Proposition~\ref{prop:TorusCover}, there exists a degree $l^3$ covering map $p_X\co \widetilde X\to X$, such that for any fiber $T\subset X$ and any component $\widetilde T$ of $p_X^{-1}(T)$, the map $p_X|_{\widetilde T}\co \widetilde T\to T$ is the covering map corresponding to $(l\mathbb Z)\times(l\mathbb Z)\subset \pi_1(T)$. By the construction of $\widetilde X$, $p_X^{-1}(T)$ has $l$ components. Let $\widetilde E=p_X^{-1}(E)$.


There is a covering map 
\begin{equation}\label{eq:qN}
q_N=q_l\times p_0\co S^1\times \widetilde N_0\to S^1\times N_0,
\end{equation}
where $q_l\co S^1\to S^1$ is the $l$--fold cyclic cover. There are $r$ components in $q_N^{-1}(S^1\times K')$, and the restriction of $q_N$ on each component is the covering map corresponding to $(l\mathbb Z)\times(l\mathbb Z)\subset \pi_1(S^1\times K')$.

Since $(l\mathbb Z)\times(l\mathbb Z)$ is a characteristic subgroup of $\mathbb Z\times\mathbb Z$, for any component $\widetilde T$ of $p_X^{-1}(T)$ and any component $\widetilde S$ of $q_N^{-1}(S^1\times K')$, the map $f\co S^1\times K'\to T$ lifts to a map $\tilde f\co \widetilde S\to \widetilde T$. Hence we can use $\tilde f$ to perform a normal connected sum of $\widetilde X$ and $S^1\times \widetilde N_0$.

Recall that  $p_X^{-1}(T)\subset \widetilde X$ has $l$ components, and $q_N^{-1}(S^1\times K')\subset S^1\times \widetilde N_0$ has $r$ components. 
Take $r$ copies of $\widetilde X$ and $l$ copies of $S^1\times \widetilde N_0$.
For any copy of $\widetilde X$ and any copy of $S^1\times \widetilde N_0$, we can perform a normal connected sum of these two manifolds along a component of $p_X^{-1}(T)$ and a component of 
$q_N^{-1}(S^1\times K')$, such that each component of $p_X^{-1}(T)$ or $q_N^{-1}(S^1\times K')$ is used exactly once. The new manifold we get, denoted by $\widetilde X_K$, is clearly an $rl^3$--fold cover of $X_K$.

By the construction, $\widetilde X_K$ is obtained by gluing $r$ copies of $\widetilde E$ and $l$ copies of $S^1\times \widetilde N$ together, such that any copy of $\widetilde E$ and any copy of $S^1\times \widetilde N$ are glued along a $T^3$. Hence there is a retraction of $\widetilde X_K$ onto $K_{r,l}$. 
\end{proof}


\section{Proof of the main theorem}\label{sect:Proof}

In this section, we will prove Theorem~\ref{thm:Main}. By Remark~\ref{rem:PosGenus}, we only consider the case that $X$ is a torus bundle over a closed surface $F$ with positive genus. 
Assume that $K$ is a nontrivial knot in $S^3$ and $X_K$ is a symplectic manifold.


\begin{lem}\label{lem:b1}
There exists a finite cover of $X_K$ with $b_1>4$.%
\end{lem}
\begin{proof}
Let $\widehat{\Sigma}\subset N_0$ be the closed surface obtained from a minimal Seifert surface $\Sigma$ of $K$ by capping off $\partial \Sigma$ with a disk. 
By \cite{G3}, $N_0$ is irreducible and $\widehat{\Sigma}$ is incompressible in $N_0$. Since $\pi_1(N_0)$ is residually finite, we can find an epimorphism $\alpha$ from $\pi_1(N_0)$ onto a finite group $G$, such that $\pi_1(\widehat{\Sigma})\not\subset\ker\alpha$. Hence $p_0\co\widetilde N_0\to N_0$, the covering map corresponding to $\ker\alpha$, is not a cyclic covering map. As a result, $p_0^{-1}(K')$ has $r>1$ components. Suppose that each component of  $p_0^{-1}(K')$ is an $l$--fold cyclic cover of $K'$. We may assume $l>5$, since we can always take a large cyclic cover of $N_0$ first.

We construct a cover $\widetilde X_K$ of $X_K$ as in Proposition~\ref{prop:Cover}.
Since there is a retraction of $\widetilde X_K$ onto $K_{r,l}$, \[b_1(\widetilde X_K)\ge b_1(K_{r,l})=(r-1)(l-1)\ge l-1>4.\hfill\qedhere\]
\end{proof}

\begin{cor}\label{cor:nontorsion}
Let $\mathfrak k$ be the canonical Spin$^c$ structure of $X_K$.
Then $c_1(\mathfrak k)$ is nontorsion.
\end{cor}
\begin{proof}
By Lemma~\ref{lem:b1}, there exists a finite cover $\widetilde X_K$ of $X_K$ with $b_1>4$.
Assume that $c_1(\mathfrak k)$ is torsion, then $c_1(\widetilde X_K)$ is also torsion since it is the pull-back of $c_1(\mathfrak k)$ by the covering map. By Theorem~\ref{thm:SymplCY},
$b_1(\widetilde X_K)\le4$, a contradiction.
\end{proof}

In order to apply Theorem~\ref{thm:Taubes}, we need the following lemma.

\begin{lem}\label{lem:b2+}
If $(r-1)(l-1)>2$, then $b_2^+(\widetilde X_K)>1$.
\end{lem}
\begin{proof}
The Euler characteristic of $X$ is zero since the fiber has zero Euler characteristic.
It is well known that the signature of $X$ is zero \cite{Meyer}. Since $X_K$ has the same homology type as $X$, both the Euler characteristic and the signature of $X_K$ are zero, and the same is true for $\widetilde X_K$. It follows that \[b_2^+(\widetilde X_K)=b_1(\widetilde X_K)-1\ge(r-1)(l-1)-1>1.\qedhere\]
\end{proof}

\begin{proof}[Proof of Theorem~\ref{thm:Main}]
Assume that $K$ is not fibered. By \cite{G3}, $N_0$ is not fibered. Let $\phi$ be the positive generator of $H^1(N_0)\cong \mathbb Z$, and let $\psi\in H^1(N)$ be the 
restriction of $\phi$. We can regard $\phi$ as a map $\pi_1(N_0)\to \mathbb Z$. 
By Theorem~\ref{thm:FV}, there exists a surjective homomorphism $\alpha\co \pi_1(N_0)\to G$, where $G$ is a finite group, such that 
\begin{equation}\label{eq:Delta0}
\Delta_{N_0}^{\alpha}=\Delta_{N_0,\phi}^{\alpha}=0.
\end{equation}

As in Proposition~\ref{prop:Cover}, let $p_0\co\widetilde N_0\to N_0$ be the covering map corresponding to $\ker \alpha$,  and let $\widetilde N=p_0^{-1}(N)$. We may assume $r>1,l>3$. Otherwise, as in the proof of Lemma~\ref{lem:b1}, we can take a regular finite cover $M_0$ of $N_0$ satisfying $r>1,l>3$, and let $\beta\co \pi_1(N_0)\to G_1$ be an epimorphism onto a finite group such that $\ker\beta=\ker\alpha\cap\pi_1(M_0)$. It follows from \cite[Lemma~2.2]{FV_JEMS} that $\Delta_{N_0}^{\beta}=0$. So we can use $\beta$ instead of $\alpha$.

Since $r>1,l>3$, we have $b_2^+(\widetilde X_K)>1$ by Lemma~\ref{lem:b2+}.

Let $(p_0)_*\co\pi_1(\widetilde N_0)\to\pi_1( N_0)$ be the induced map on $\pi_1$, and let \[\widetilde{\phi}=\phi\circ(p_0)_*\co\pi_1(\widetilde N_0)\to\mathbb Z.\]
Let $\widetilde{\phi}_*\co \mathbb Z[H(\widetilde N_0)]\to\mathbb Z[\mathbb Z]$ be the induced ring homomorphism.
By Proposition~\ref{prop:AlexHom}, (\ref{eq:Delta0}) implies  $(p_0)_*(\Delta_{\widetilde N_0})=0$, hence
\begin{equation}\label{eq:phi=0}
\widetilde{\phi}_*(\Delta_{\widetilde N_0})=0.
\end{equation}

Let $(p_0|\widetilde N)_*\co\pi_1(\widetilde N)\to\pi_1( N)$ be the induced map on $\pi_1$, and let \[\widetilde{\psi}=\psi\circ(p_0|\widetilde N)_*\co \pi_1(\widetilde N)\to\mathbb Z.\]
Let $\widetilde{\psi}_*\co \mathbb Z[H(\widetilde N)]\to\mathbb Z[\mathbb Z]$ be the induced ring homomorphism. We also regard $\widetilde{\psi}$ as a cohomology class in $H^1(\widetilde N)$, then $\widetilde{\psi}\in H^1(\widetilde N)$ is the pull-back of $\psi \in H^1(N)$ by the covering map.
Clearly, for any component $\widetilde{K'}$ of $p_0^{-1}(K')$, we have $\widetilde{\psi}([\widetilde{K'}])\ne0$. Hence we can use
Corollary~\ref{cor:AlexGlue} and (\ref{eq:phi=0}) to conclude
\begin{equation}\label{eq:psi=0}
\widetilde{\psi}_*(\Delta_{\widetilde N})=0.
\end{equation}

We construct a finite cover $\widetilde X_K$ as in Proposition~\ref{prop:Cover}.
Suppose that $\omega$ is a symplectic form on $X_K$. Since $[T]\ne0\in H_2(X;\mathbb R)\cong H_2(X_K;\mathbb R)$ and $c_1(\mathfrak k)\ne0\in H^2(X_K;\mathbb R)$ by Corrollary~\ref{cor:nontorsion}, we may
perturb and rescale $\omega$ so that 
\begin{equation}\label{eq:omega}
[\omega]([T])\ne0,\quad c_1(\mathfrak k)\smile[\omega]\ne0,
\end{equation}
 and $[\omega]\in H^2(X_K;\mathbb Z)$.
Let $\Omega$ be the pull-back of $\omega$ on $\widetilde X_K$, then $\Omega$ is also a symplectic form.
Moreover, it follows from (\ref{eq:omega}) that 
\begin{equation}\label{eq:Omega}
[\Omega]([\widetilde T])\ne0,\quad c_1(\widetilde X_K,\Omega)\smile[\Omega]\ne0,
\end{equation}

The inclusion map $S_1\times N\subset X_K$ induces a map $$\iota_N^*\co H^2(X_K)\to H^2(S^1\times N)\cong H^1(S^1)\otimes H^1(N).$$
Let $\sigma$ be the positive generator of $H^1(S^1)$.  Then 
\begin{equation}\label{eq:Pullback}
\iota_N^*([\omega])=k\sigma\otimes \psi, \text{  for some integer }k\ne0,
\end{equation}
by (\ref{eq:omega}).

Let $X_2\subset \widetilde X_K$ be a copy of $S^1\times\widetilde N$,
let $X_1=\widetilde X_K\setminus \mathrm{int}(X_2)$, and $M=\partial X_1$.
Let $\iota_i^*\co H^2(\widetilde X_K)\to H^2(X_i)$, $i=1,2$, be the natural maps induced by the inclusion maps.

Let $p^*\co H^2(\widetilde N,\partial \widetilde N)\to H^2(S^1\times\widetilde N,S^1\times\partial \widetilde N)$ be the map induced by the projection. Let $q_N$ be the covering map in (\ref{eq:qN}). If $w\in H^2(\widetilde N,\partial \widetilde N)$, using (\ref{eq:Pullback}), we have
\begin{eqnarray}
\iota_2^*[\Omega]\smile p^*(w)&=&q_N^*(\iota_N^*[\omega])\smile p^*(w)\nonumber\\
&=&q_N^*(k\sigma\otimes \psi)\smile p^*(w)\nonumber\\
&=&kl\sigma\widetilde{\psi}\smile p^*(w)\nonumber\\
&=&kl\widetilde{\psi}\smile_3 w,\label{eq:cup3}
\end{eqnarray}
where $\smile_3$ means the cup product in $(\widetilde N,\partial\widetilde N)$.
Here we identify an element $a\cup b\in H^n(Y^n,\partial Y^n)$ with an integer via the isomorphism $ H^n(Y^n,\partial Y^n)\cong\mathbb Z$. 

Let 
$$\rho_i\co H^2(X_i,\partial X_i)\to H^2(\widetilde X_K), \quad i=1,2,$$
be the natural restriction maps, and let
$$\rho=\rho_1+\rho_2\co H^2(X_1,\partial X_1)\oplus H^2(X_2,\partial X_2)\to H^2(\widetilde X_K).$$

Suppose that $z_i\in H^2(X_i,\partial X_i)$, $i=1,2$,  then it is elementary to check
\begin{equation}\label{eq:CupDecomp}
\rho(z_1,z_2)\smile [\Omega]
=z_1\smile \iota_1^*[\Omega]+z_2\smile \iota_2^*[\Omega].
\end{equation}

Suppose that $n=c_1(\widetilde X,\Omega)\smile[\Omega]$. Then $n\ne0$ by (\ref{eq:Omega}). Using Theorem~\ref{thm:Gluing3} and (\ref{eq:CupDecomp}), we have
\begin{eqnarray}
&&\sum_{z\in H^2(\widetilde X_K), z\smile[\Omega]=n}SW_{\widetilde X_K}(z)\nonumber\\
&=&\sum_{\substack{z_1\in H^2(X_1,\partial X_1)\\ z_2\in H^2(X_2,\partial X_2)\\ z_1\smile \iota_1^*[\Omega]+z_2\smile \iota_2^*[\Omega]=n}}SW_{X_1}(z_1)SW_{X_2}(z_2)\nonumber\\
&=&\sum_{z_1\in H^2(X_1,\partial X_1)}SW_{X_1}(z_1)\cdot
\left(\sum_{\substack{z_2\in H^2(X_2,\partial X_2)\\ z_2\smile {\iota}_2^*[\Omega]=n-z_1\smile {\iota}_1^*[\Omega]}}SW_{X_2}(z_2)\right).\label{eq:DoubSum}
\end{eqnarray}
It follows from (\ref{eq:cup3}) and Theorem~\ref{thm:MT} that the inner sum in (\ref{eq:DoubSum}) is a coefficient in $\widetilde{\psi}_*(\Delta_{\widetilde N})$, which is zero by (\ref{eq:psi=0}). Hence the right hand side of (\ref{eq:DoubSum}) is zero.
This contradicts Theorem~\ref{thm:Taubes} and the fact that $n\ne0$.
\end{proof}

\end{document}